\documentclass[reqno]{amsart}

\usepackage{setspace}

\usepackage{amsmath}
\usepackage{amscd,amsthm,amsfonts,amsopn,amssymb}
\usepackage{epsfig,hyperref}

\newtheorem{theorem}{Theorem}[section]

\newtheorem{proposition}[theorem]{Proposition}

\newtheorem{definition}[theorem]{Definition}
\newtheorem{lemma}[theorem]{Lemma}

\theoremstyle{remark}

\DeclareMathOperator{\divop}{div}
\DeclareMathOperator{\curl}{curl}

\allowdisplaybreaks

\begin{document}

\title{Non-existence of splash singularities for the two-fluid Euler--Navier-Stokes system}
\author{Aynur Bulut}
\date{\today}

\begin{abstract}
We consider a system of two incompressible fluids separated by a free interface.  The first fluid is inviscid, governed by the
Euler system, while the second fluid has positive viscosity and is governed by the Navier-Stokes system.  We formulate a notion
of splash-type self-intersection singularities, and show that this system cannot form such a singularity in finite time.  The main
obstacle in our analysis is to handle the effect of bulk vorticity induced by the Navier-Stokes flow.  To the best of our knowledge
this is the first result on preclusion of splash-type singularities in the presence of non-trivial vorticity interior to one
of the fluids.
\end{abstract}

\maketitle

\section{Introduction}

In this work, we consider the free interface evolution associated to a coupled Euler--Navier-Stokes system of two incompressible 
two-dimensional fluids, each occupying a simply connected region in the plane.  

To model the two fluid system, we let $I\subset\mathbb{R}$ be a time interval, and 
let $$z\in C^2(I\times \mathbb{R};\mathbb{R}^2),$$ be an injective function parameterizing the 
interface $\Gamma_t=\{z(t,x):x\in\mathbb{R}\}$ between the two fluid regions.  We suppose that for each $t\in I$, we can partition 
the plane as $$\mathbb{R}^2=\Omega_{E}(t)\cup \Gamma_t\cup \Omega_{NS}(t),$$ where $\Omega_E(t)$ is the region occupied 
by the Euler fluid at time $t$, and $\Omega_{NS}(t)$ is the region occupied by the Navier-Stokes fluid.  

We consider the evolution of velocity vector fields $u$, associated to the Euler flow, and $v$, associated 
to the Navier-Stokes flow, defined in the domains $\Omega_{E}(t)$ and $\Omega_{NS}(t)$, respectively, and satisfying
\begin{align}
\left\lbrace\begin{array}{l}\rho_E(\partial_t u+(u\cdot\nabla)u)=-\nabla p-g\rho_E\mathbf{e}_2,\\
\hspace*{2.5in} \textrm{in}\,\, \{(t,x):t\in I, x\in \Omega_E(t)\},\\ \ \\
\rho_{NS}(\partial_t v+(v\cdot\nabla)v)=\nu_{NS}\Delta v-\nabla q-g\rho_{NS}\mathbf{e}_2,\\
\hspace*{2.5in} \textrm{in}\,\, \{(t,x):t\in I, x\in \Omega_{NS}(t)\},\end{array}\right.\label{eqE1}
\end{align}
along with the incompressibility conditions
\begin{align}
\divop_x u(t,x)=0\,\,\textrm{and}\,\,\divop_x v(t,x)=0.\label{eqE2}
\end{align}
where the constant $g>0$ represents the strength of the gravitational potential, $\nu_{NS}>0$ denotes the viscosity coefficient 
for the Navier-Stokes fluid, and $\rho_{E}>0$ and $\rho_{NS}>0$ are constants denoting the density of the Euler and 
Navier-Stokes fluids, respectively.  Moreover, we assume $\curl_x u(t,x)=0$ on $\Omega_{E}(t)$.  

Our study is motivated by a number of recent results on self-intersection singularities for moving boundary problems in fluid 
dynamics (these results arise in the context of an extensive theory for the fully nonlinear system governing the evolution
of water waves; see work of Wu \cite{Wu} as well as Ionescu-Pusateri \cite{IP} and Ifrim-Tataru \cite{IT}, and the references cited
in these works).  Studying the possible formation of splash-type and splat-type
singularities, in the work of Castro, C\'ordoba, Fefferman, Gancedo, and
G\'omez-Serrano \cite{CCFGGS,CCFGGS-2}, it was shown that these types of self-intersections
can arise dynamically for the free boundary problem associated 
to the flow of a single fluid (here, the region outside of the domain of the fluid is considered as a vacuum, with or without 
the effects of surface tension).  See also the work of Coutand and Shkoller \cite{CS-2}.  On the other hand, in 
\cite{FIL-NoSplash} and \cite{CS}, Fefferman-Ionescu-Lie and Coutand-Shkoller, respectively, applied a related analysis to the situation of a
two-fluid coupled system, where the 
vacuum region is replaced by the flow of a fluid (of possibly different density).  In this case, the presence of the second 
fluid {\it prevents} the formation of splash-type intersections.

The results of \cite{CCFGGS,CCFGGS-2} and \cite{CS-2} were initially shown in the setting of the Euler equation.  Subsequent 
analysis extended the construction of singularity to the setting of viscous fluids modeled by the Navier-Stokes equations, 
see, e.g. \cite{CCFGGS-3} and \cite{CS-3}.  The goal of the present paper is to consider the preclusion of splash-type
singularities for a two-fluid system in the spirit of \cite{FIL-NoSplash}, in the case that one of the fluids is driven by the
Navier-Stokes system, and therefore acquires non-trivial vorticity in its interior.  This additional vorticity is the main 
additional effect that must be considered in our analysis.

In the rest of this paper, we specialize to the case of a periodic setting; that is, we assume that $z$, $u$ and $v$ are periodic 
in the $\mathbf{e}_1$ direction with period $2\pi$ .  We also use the convention that
\begin{align}
\widetilde{u}(t,\alpha)=u(t,z(t,\alpha))\quad\textrm{and}\quad \widetilde{v}(t,\alpha)=v(t,z(t,\alpha))\label{def-ut-vt}
\end{align}
denotes the velocities $u$ and $v$ evaluated on the interface.

To specify the boundary conditions, we set 
\begin{align*}
{\mathbf n}(t,\alpha)=-[(\partial_\alpha z)(t,\alpha)]^\perp.
\end{align*}
For $u$, we require 
\begin{align}
\Big(\partial_t z(t,\alpha)-u(t,z(t,\alpha))\Big)\cdot {\mathbf n}(t,\alpha)=0,\label{eqE3}
\end{align}
while for $v$ we impose the boundary condition
\begin{align}
\partial_t z(t,\alpha)-v(t,z(t,\alpha))=0.\label{eqE4}
\end{align}

The equations and boundary conditions ($\ref{eqE1}$)--($\ref{eqE4}$) are supplemented by a compatibility condition on the interface.
In particular, fixing a coefficient of surface tension $\sigma\geq 0$, we also require that the evolution of the interface satisfies
\begin{align}
\nonumber &p(t,z(t,\alpha)){\mathbf n}(t,\alpha)+\Big(q(t,z(t,\alpha))1_{2\times 2}-2\nu_{NS}D(v)|_{(t,z(t,\alpha))}\Big){\mathbf n}(t,\alpha)\\
&\hspace{3.2in}=-\sigma K_{z(t,\cdot)}(\alpha){\mathbf n}(t,\alpha),\label{eqE5}
\end{align}
where the expressions
\begin{align*}
D(f)=\nabla f+(\nabla f)^\top,\quad K_{z(t,\cdot)}(\alpha):=\frac{(\partial_{\alpha}z_1)(\partial_\alpha^2z_2)-(\partial_\alpha z_2)(\partial_\alpha^2 z_1)}{[(\partial_\alpha z_1)^2+(\partial_\alpha z_2)^2]^{3/2}}\bigg|_{(t,\alpha)}
\end{align*}
denote the symmetrized gradient and the curvature of the interface described by $\alpha\mapsto z(t,\alpha)$ at the point $z(t,\alpha)$.  In all expressions above, if $x=(x_1,x_2)\in\mathbb{R}^2$, then $x^\perp=(-x_2,x_1)$.

To give the formal statement of our main result, we introduce some auxiliary quantities which will be used to identify the specific geometric splash situation which we seek to exclude.  The fundamental definition in this context is a chord-arc functional, defined as follows: given $z:\mathbb{R}\times I\rightarrow\mathbb{R}$, define 
\begin{align*}
F_z(t,\alpha,\beta)=\left\lbrace\begin{array}{ll}\frac{|z(t,\alpha)-z(t,\beta)|}{|\alpha-\beta|},&\quad\textrm{if}\,\,\alpha\neq \beta,\\
|(\partial_\alpha z)(t,\alpha)|,&\quad\textrm{if}\,\, \alpha=\beta.\end{array}\right.
\end{align*}

Now, for each $z:\mathbb{R}\times I\rightarrow\mathbb{R}$, set
\begin{align*}
CA(t,z):=\inf_{\alpha,\beta\in \mathbb{R}} |F_z(t,\alpha,\beta)|
\end{align*}
and
\begin{align*}
CA_T(z):=\inf_{t\in [0,T]} CA(t,z).
\end{align*}

With these definitions in hand, we are now ready to turn to the particular notion of solution that will be used to establish
our results.  In particular, we give a notion of {\it admissible solution} to the system ($\ref{eqE1}$)--($\ref{eqE5}$), which in
addition to being a smooth solution of the system of PDEs, satisfies a class of additional a priori bounds (what happens in lower-regularity
situations is an area of active research; see, for instance \cite{CEG}).

\begin{definition}[Notion of ($A$, $\textrm{E}_{\textrm{wild}}$, $\textrm{NS}_{\textrm{tame}}$)--admissible solutions]
\label{def-ad}
For $T>0$ and $A>0$, we say that $(u,v,z,p,q)$ is an ($A$, $\textrm{E}_{\textrm{wild}}$, $\textrm{NS}_{\textrm{tame}}$)--admissible 
solution to the system ($\ref{eqE1}$)--($\ref{eqE5}$) if it is a smooth solution to the system satisfying $CA(0,z)\geq 1/A$, $\lVert \widetilde{u}(0,\cdot)\rVert_{L^\infty(\mathbb{R})}\leq A$, $\lVert v(0)\rVert_{C^3(\Omega_{NS}(0))}\leq A$, and 
\begin{align}
\lVert z(t,\alpha)-(\alpha,0)\rVert_{C^4_{\alpha,t}([0,T]\times \mathbb{R})}\leq A,\label{eq-assumption}
\end{align}
together with
\begin{align}
\lVert v\rVert_{C^2(\{(t,x)\in [0,T]\times\mathbb{R}^2:x\in \overline{\Omega_{NS}(t)}\})}\leq A,\label{eq-hypB}
\end{align}
and
\begin{align}
\inf\{|z_\alpha(t,\alpha)|:\alpha\in\mathbb{R},t\in [0,T]\}&\geq \frac{1}{A}.\label{eq-hypC}
\end{align}
\end{definition}

Our main result is the following theorem, which asserts that a splash-type singularity cannot form in finite time.

\begin{theorem}
\label{thm1}
Fix $T>0$ and $A>0$.  Then there exists $C_0>0$ such that if $(u,v,z,p,q)$ is an ($A$, $\textrm{E}_{\textrm{wild}}$, $\textrm{NS}_{\textrm{tame}}$)--admissible solution to the system ($\ref{eqE1}$)--($\ref{eqE5}$), then the condition $CA_T(z)>0$ implies $CA_T(z)\geq C_0$.
\end{theorem}

The main ingredients in the proof of Theorem \ref{thm1} are a family of a priori estimates for the relevant vorticities, combined with an analysis of the evolution of the interface near a potential splash point.  We make use of a geometric reduction first introduced in \cite{FIL-NoSplash} to reduce the analysis to a particularly simple setting.  The argument then reduces to the analysis of the contributions of two terms which control the formation of a potential point of self-intersection.  The first term arises from the vorticity concentrated on the interface, which is estimated as in \cite{FIL-NoSplash} by estimating the resulting singular integrals.  The second contribution, which is novel in our setting, is the contribution of the bulk vorticity interior to the Navier-Stokes flow.  

To the best of our knowledge, this is the first result on {\it preclusion} of splash-type singularities in which non-trivial vorticity plays a role.  While we do not discuss the existence of weak solutions to this system, or issues such as weak-strong uniqueness, there have been a number of recent works which address these points for related systems (see, for instance \cite{FH}).


\subsection*{Outline of the paper}

In Section $2$ below we discuss a preliminary reformulation of the system of equations ($\ref{eqE1}$)--($\ref{eqE5}$), and give the precise statement of our main results.  In Section $3$, we obtain an a priori bound for the viscosity and recall some preliminary arguments on the nature of a potential splash singularity from \cite{FIL-NoSplash}.  In Section $4$ we establish our main result, Theorem $\ref{thm1}$.  A brief appendix contains some supporting arguments and concluding remarks.

\subsection*{Acknowledgments}

The author is very grateful to Charlie Fefferman for many valuable conversations during the preparation of this work.

\section{Derivation of the free interface equations}

In this section, we rewrite the system ($\ref{eqE1}$)--($\ref{eqE5}$) in terms of a system posed on the closure of the domain of the Navier-Stokes fluid, given by $$\overline{\Omega_{NS}(t)}=\Omega_{NS}(t)\cup \Gamma_t$$ for $t>0$.  This is based on a study of the evolution of the interface vorticity $\omega_\Gamma$, 
defined by
\begin{align*}
\omega_\Gamma(t,\alpha)&:=\Big(\widetilde{u}(t,\alpha)-\widetilde{v}(t,\alpha)\Big)\cdot (\partial_\alpha z)(t,\alpha),
\end{align*}
for $t\in I$ and $\alpha\in\mathbb{R}$, where $\widetilde{u}$ and $\widetilde{v}$ are, as defined in ($\ref{def-ut-vt}$), the 
parameterized restrictions of $u$ (the velocity of the Euler fluid) and $v$ (the velocity of the Navier-Stokes fluid) to 
the interface $\Gamma_t$.

We interpret the model of Section 1 in the sense of asking that, for a given triple $(v,q,z)$, the pair $(u,p)$ should be a (weak) solution of
\begin{align*}
\left\lbrace\begin{array}{ll}\rho_E(\partial_t u+(u\cdot\nabla)u)=-\nabla p-g\rho_Ee_2&\quad\textrm{in}\,\,\Omega_{E}(t),\\
\divop u=\curl u=0&\quad\textrm{in}\,\,\Omega_{E}(t),\\
(\partial_t z-u(t,z(t,\alpha)))\cdot {\mathbf n}=0&\quad\textrm{on}\,\,\Gamma,\\
(p-q)\cdot (-{\mathbf n})-2\nu_{NS}D(v)|_{t,z(t,\alpha)}(-{\mathbf n})=-\sigma K{\mathbf n}&\quad\textrm{on}\,\,\Gamma.
\end{array}\right.
\end{align*}

For notational convenience, we let $\omega_v$ denote the vorticity internal to the Navier-Stokes fluid, defined by
\begin{align*}
\omega_v(t,x)=\curl(v)\chi_{\Omega_{NS}}(t).
\end{align*}
for $t\in I$ and $x\in\mathbb{R}^2$.

Two important terms in the resulting equations are the Birkhoff-Rott integral associated to the interface vorticity $\omega_\Gamma$
\begin{align*}
B_\Gamma(z,\omega_\Gamma;t,\alpha)&:=(2\pi)^{-1}p.v.\int_{\mathbb{R}} \frac{(z(t,\alpha)-z(t,\beta))^\perp}{|z(t,\alpha)-z(t,\beta)|^2}\omega_\Gamma(t,\beta)d\beta,
\end{align*}
and the corresponding contribution of the ``bulk'' vorticity $\omega_v$,
\begin{align*}
B_{\textrm{NS}}(z,\omega_v;t,\alpha)&:=\int_{\Omega_{NS}(t)} \frac{(z(t,\alpha)-y)^\perp}{|z(t,\alpha)-y|^2}\omega_v(t,y)dy.
\end{align*}
The relevance of these operators is shown by the next proposition.

\begin{proposition}
\label{prop1}
Fix $T>0$, suppose $(u,v,z,p,q)$ solves the system ($\ref{eqE1}$)--($\ref{eqE5}$) on the time interval $I=[0,T]$, and that each of these functions is periodic in the ${\mathbf e}_1$ direction (with period $2\pi$).  Suppose $z\in C^4(I\times \mathbb{R};\mathbb{R}^2)$ and that $u$ and $v$ belong to appropriate function spaces.  

Then, with $\tilde{u}$ as in (\ref{def-ut-vt}) and  $\omega_\Gamma$, $\omega_v$, $B_{\Gamma}$ and $B_{\textrm{NS}}$ as above, we have
\begin{align}
\widetilde{u}(t,\alpha)&=\frac{\omega_\Gamma\, \partial_\alpha z}{2|\partial_\alpha z|^2}\bigg|_{(t,\alpha)}+B_\Gamma(z,\omega_\Gamma;t,\alpha)+B_{\textrm{NS}}(z,\omega_v;t,\alpha),\label{eq1}
\end{align}
\end{proposition}

\begin{proof}
Fix $t\in I$, and note that by the Biot-Savart law, one has
\begin{align*}
u(t,x)=(2\pi)^{-1}p.v.\int_{\mathbb{R}} \frac{(x-z(t,\beta))^\perp}{|x-z(t,\beta)|^2}\omega_\Gamma(t,\beta)d\beta+\int_{\Omega_{NS}(t)} \frac{(x-y)^\perp}{|x-y|^2}\omega_v(t,y)dy
\end{align*}
for $x\in \Omega_{E}(t)$.  By the regularity of the domain $\Omega_{E}(t)$, standard limiting arguments for functions defined via contour integrals give
\begin{align*}
&\lim_{\epsilon\downarrow 0}\bigg\lVert u(t,z(t,\alpha)-\epsilon {\mathbf n}(t,\alpha))\\
&\hspace{0.8in}-\bigg(\frac{\omega_\Gamma\, \partial_\alpha z}{2|\partial_\alpha z|^2}\bigg|_{(t,\alpha)}+B_\Gamma(z,\omega_\Gamma;t,\alpha)+B_{\textrm{NS}}(z,\omega_v;t,\alpha)\bigg)\bigg\rVert_{L_{\alpha}^\infty(\mathbb{R})}=0.
\end{align*}
\end{proof}

The reduction of the equation for $u$ to an equation for $\widetilde{u}$ is given by the following proposition. 
In the statement and proof, as well as in the remainder of this paper, we use the notation 
$$ \widetilde{p}(t,\alpha):=p(t,z(t,\alpha)),\quad \widetilde{q}(t,\alpha):=q(t,z(t,\alpha)),$$
when $t\in I$ and $\alpha\in\mathbb{R}$.

\begin{proposition}
\label{prop1b}
Fix $T>0$ and let $(u,v,z,p,q)$ be as in Proposition \ref{prop1}.  Then
\begin{align}
\partial_t(\widetilde{u}\cdot \partial_\alpha z)-\partial_\alpha(\widetilde{u}\cdot \partial_t z)+\frac{1}{2}\partial_\alpha[|\widetilde{u}|^2]+\frac{1}{\rho_E}\partial_\alpha\widetilde{p}+g\partial_\alpha z_2\bigg|_{(t,\alpha)}=0\label{eq3}
\end{align}
holds for all $t\in I$ and $\alpha\in\mathbb{R}$.
\end{proposition}

\begin{proof}
We compute 
\begin{align*}
\partial_t(\widetilde{u}\cdot\partial_\alpha z)=\sum_{j=1}^2 (\partial_tu_j)(\partial_\alpha z_j)+(\partial_1u_j)(\partial_t z_1)(\partial_\alpha z_j)+(\partial_2u_j)(\partial_t z_2)(\partial_\alpha z_j)+(u_j)(\partial_{t,\alpha} z_j),
\end{align*}
where each instance of $u$ is evaluated at $(t,z(t,\alpha))$, and, with the same convention,
\begin{align*}
-\partial_\alpha(\widetilde{u}\cdot\partial_tz)=-\sum_{j=1}^2 (\partial_1u_j)(\partial_\alpha z_1)(\partial_t z_j)+(\partial_2u_j)(\partial_\alpha z_2)(\partial_t z_j)-(u_j)(\partial_{\alpha,t}z_j),
\end{align*}
together with
\begin{align*}
\frac{1}{2}\partial_\alpha[|\widetilde{u}|^2]&=\sum_{j=1}^2 (u_j)(\partial_\alpha u_j)\\
&=\sum_{j=1}^2 (u_j)(\partial_1 u_j)(\partial_\alpha z_1)+(u_j)(\partial_2 u_j)(\partial_\alpha z_2).
\end{align*}

Assembling these, 
\begin{align*}
&\partial_t(\widetilde{u}\cdot\partial_\alpha z)-\partial_\alpha(\widetilde{u}\cdot \partial_tz)+\frac{1}{2}\partial_\alpha[|\widetilde{u}|^2]\\
&\hspace{0.2in}=\bigg(\sum_{j=1}^2 (\partial_t u_j)(\partial_\alpha z_j)+(u_j)(\partial_1 u_j)(\partial_\alpha z_1)+(u_j)(\partial_2 u_j)(\partial_\alpha z_2)\bigg)\\
&\hspace{0.4in}+(\partial_1u_2)(\partial_t z_1)(\partial_\alpha z_2)+(\partial_2 u_1)(\partial_t z_2)(\partial_\alpha z_1)\\
&\hspace{0.4in}-(\partial_1 u_2)(\partial_\alpha z_1)(\partial_t z_2)-(\partial_2 u_1)(\partial_\alpha z_2)(\partial_t z_1).
\end{align*}
Since $\curl u=\partial_2u_1-\partial_1 u_2=0$, the second and third lines in the above expression sum to $0$.  

Next, using the equation for evolution of $u$ (evaluated on the interface $\Gamma$), the left-hand side of the above expression is equal to
\begin{align*}
&\sum_{j=1}^2 \bigg(-\big[\sum_{k=1}^2 (u_k)(\partial_k u_j)\big]-\rho_E^{-1}(\partial_jp)-g\delta_{j,2}\bigg)(\partial_\alpha z_j)\\
&\hspace{0.6in}+(u_j)(\partial_1u_j)(\partial_\alpha z_1)+(u_j)(\partial_2u_j)(\partial_\alpha z_2)=-\rho_E^{-1}\partial_\alpha\widetilde{p}-g(\partial_\alpha z_2),
\end{align*}
where we have again used $\curl u=0$, and in addition noted that $\partial_\alpha \widetilde{p}=(\partial_1p)(\partial_\alpha z_1)+(\partial_2p)(\partial_\alpha z_2)$ with notational conventions as above.  This gives the desired identity.
\end{proof}

To conclude this section, we collect the above with the interface condition (\ref{eqE3})
to obtain an equivalent form of the system of PDEs, which will form the basis for our study in the rest of 
the paper.  In particular, for unknowns $(\omega_\Gamma, v,z,\widetilde{p},q)$, the system described 
above is equivalent to
\begin{align*}
\left\lbrace\begin{array}{ll}
(\partial_t z(t,\alpha)-\widetilde{u}(t,\alpha))\cdot \mathbf{n}_2(t,\alpha)=0,\\
(\partial_t (\widetilde{u}\cdot\partial_\alpha z)-\partial_\alpha(\widetilde{u}\cdot \partial_t z)+\frac{1}{2}\partial_\alpha[|\widetilde{u}|^2]+\frac{1}{\rho_E}\partial_\alpha \widetilde{p}+g\partial_\alpha z_2)|_{t,\alpha}=0,\\
\rho_{NS}(\partial_t v+(v\cdot\nabla)v)=\nu_{NS}\Delta v-\nabla q-g\rho_{NS}{\mathbf e}_2,\\
\partial_t z(t,\alpha)-v(t,z(t,\alpha))=0,\\
(\widetilde{p}+\widetilde{q}){\mathbf n}-2\nu_{NS}D(v)|_{t,z}{\mathbf n}=-\sigma K_z(\alpha){\mathbf n},\\
\divop(v)=0,
\end{array}\right.
\end{align*}
with formulas for $\widetilde{u}$ and $\omega_\Gamma$ given above (depending on the above unknowns).  Here, the first second, fourth, and fifth lines are equations on the interface (i.e. for $t>0$ and $\alpha\in \mathbb{R}$), while the third and sixth lines are equations in the interior of the Navier-Stokes fluid.

\section{Bounds on the vorticity and initial characterization of a splash point}

In this section, we establish a priori bounds on $\omega_\Gamma$, the vorticity concentrated on the boundary, and formulate a geometric
reduction which rescales and orients a candidate splash point to a generic position amenable to our analysis in the rest of this paper.

We begin by observing that the condition (\ref{eq-hypB}) in the definition of $(A,\textrm{E}_{\textrm{wild}},\textrm{NS}_{\textrm{tame}})$--admissibility implies
\begin{align*}
\sup_{t\in [0,T]} \lVert \omega_v\rVert_{L_x^p(\Omega_{NS}(t))}\leq C,\quad 1\leq p\leq \infty.
\end{align*}

Our a priori estimate for $\omega_\Gamma$ is then as follows:

\begin{proposition}
\label{lem0}
Fix $T>0$ and $A>0$, and suppose that $(\omega_\Gamma,v,z,\widetilde{p},q)$ is an ($A$, $\textrm{E}_{\textrm{wild}}$, $\textrm{NS}_{\textrm{tame}}$)--admissible solution of the system ($\ref{eqE1}$)--($\ref{eqE5}$).  Then there exists $C>0$ such that $$\lVert \omega_\Gamma\rVert_{L^\infty([0,T]\times \mathbb{R})}\leq C.$$
\end{proposition}

A key tool in the proof of Proposition $\ref{lem0}$ is the following lemma, which identifies an equation for $\widetilde{v}$, the restriction to the interface of the Navier-Stokes flow, which is analogous to the equation (\ref{eq3}) we derived for $\widetilde{u}$ in Proposition \ref{prop1b} of the previous section.

\begin{lemma}
\label{lem-vtilde}
Let $T$, $A$, $(\omega_\Gamma,v,z,\widetilde{p},q)$ be as in Proposition \ref{lem0}.  Then
\begin{align}
\nonumber &\partial_t(\widetilde{v}\cdot \partial_\alpha z)-\partial_\alpha(\widetilde{v}\cdot \partial_t z)+\frac{1}{2}\partial_\alpha[|\widetilde{v}|^2]\\
&\hspace{1.2in}+\frac{1}{\rho_{NS}}\partial_\alpha\widetilde{q}+g\partial_\alpha z_2-\frac{\nu_{NS}}{\rho_{NS}}(\Delta v)\cdot (\partial_\alpha z)\bigg|_{(t,\alpha)}=0\label{eq3b}
\end{align}
holds for all $t\in I$, $\alpha\in\mathbb{R}$.
\end{lemma}

\begin{proof}
We argue similarly as in the proof of Proposition \ref{prop1b}, while noting that the curl of $v$ does not necessarily vanish in $\Omega_{NS}(t)$.  In particular, as before
\begin{align*}
\partial_t(\widetilde{v}\cdot\partial_\alpha z)&=\sum_{j=1}^2 \bigg[(\partial_tv_j)(\partial_\alpha z_j)\\
&\hspace{0.2in}+(\partial_1v_j)(\partial_t z_1)(\partial_\alpha z_j)+(\partial_2v_j)(\partial_t z_2)(\partial_\alpha z_j)+(v_j)(\partial_{t,\alpha} z_j)\bigg],
\end{align*}
and
\begin{align*}
-\partial_\alpha(\widetilde{v}\cdot\partial_tz)=-\sum_{j=1}^2 (\partial_1v_j)(\partial_\alpha z_1)(\partial_t z_j)+(\partial_2v_j)(\partial_\alpha z_2)(\partial_t z_j)-(v_j)(\partial_{\alpha,t}z_j),
\end{align*}
where both of these equalities are understood with the convention that $v$ is evaluated at $(t,z(t,\alpha))$.

Moreover, also as before,
\begin{align*}
\frac{1}{2}\partial_\alpha[|\widetilde{v}|^2]&=\sum_{j=1}^2 (v_j)(\partial_1 v_j)(\partial_\alpha z_1)+(v_j)(\partial_2 v_j)(\partial_\alpha z_2).
\end{align*}

Recalling that $\curl v=\partial_2 v_1-\partial_1v_2$, this leads to
\begin{align*}
&\partial_t(\widetilde{v}\cdot\partial_\alpha z)-\partial_\alpha(\widetilde{v}\cdot \partial_tz)+\frac{1}{2}\partial_\alpha[|\widetilde{v}|^2]\\
&\hspace{0.2in}=\bigg(\sum_{j=1}^2 (\partial_t v_j)(\partial_\alpha z_j)+(v_j)(\partial_1 v_j)(\partial_\alpha z_1)+(v_j)(\partial_2 v_j)(\partial_\alpha z_2)\bigg)\\
&\hspace{0.4in}-(\curl v)\Big((\partial_t z_1)(\partial_\alpha z_2)-(\partial_t z_2)(\partial_\alpha z_1)\Big).
\end{align*}
Now using the equations for $\partial_tv_j$ and $\partial_t z_j=v_j$, $j=1,2$, the left-hand side of the above expression is equal to
\begin{align*}
&\sum_{j=1}^2 \bigg[\bigg(-\big[\sum_{k=1}^2 (v_k)(\partial_k v_j)\big]+\frac{\nu_{NS}}{\rho_{NS}}\Delta v_j-\rho_{NS}^{-1}(\partial_jq)-g\delta_{j,2}\bigg)(\partial_\alpha z_j)\\
&\hspace{0.6in}+(v_j)(\partial_1v_j)(\partial_\alpha z_1)+(v_j)(\partial_2v_j)(\partial_\alpha z_2)\bigg]\\
&\hspace{0.2in}-(\curl v)\Big((v_1)(\partial_\alpha z_2)-(v_2)(\partial_\alpha z_1)\Big)\\
&\hspace{0.4in}=\sum_{j=1}^2 \frac{\nu_{NS}}{\rho_{NS}}(\Delta v_j)(\partial_\alpha z_j)-\rho_{NS}^{-1}\partial_\alpha\widetilde{q}-g\partial_\alpha z_2,
\end{align*}
where a direct computation shows that the terms of the form $(v_i)(\partial_j v_k)(\partial_\alpha z_\ell)$ collectively sum to zero, and where we have used that, via the chain rule, $\partial_\alpha \widetilde{q}=(\partial_1q)(\partial_\alpha z_1)+(\partial_2q)(\partial_\alpha z_2)$.
\end{proof}

We now return to the proof of the proposition.

\begin{proof}[Proof of Proposition $\ref{lem0}$]
We proceed as in the proof of Lemma $2.2$ in \cite{FIL-NoSplash}.  Note that Proposition $\ref{prop1b}$ and Lemma $\ref{lem-vtilde}$ imply that $\omega_\Gamma$ satisfies 
\begin{align*}
\partial_t\omega_\Gamma&=\partial_t (\widetilde{u}\cdot\partial_\alpha z)-\partial_t(\widetilde{v}\cdot\partial_\alpha z)\\
&=\partial_t \Big[(\widetilde{u}-\frac{\rho_{NS}}{\rho_E}\widetilde{v})\cdot\partial_\alpha z\Big]-\Big(\frac{\rho_E-\rho_{NS}}{\rho_E}\Big)\partial_t(\widetilde{v}\cdot\partial_\alpha z)\\
&=\partial_\alpha(\widetilde{u}\cdot \partial_t z)-\frac{1}{2}\partial_{\alpha}[|\widetilde{u}|^2]-\frac{1}{\rho_E}\partial_{\alpha}\widetilde{p}-g\partial_\alpha z_2\\
&\hspace{0.2in}-\frac{\rho_{NS}}{\rho_E}\Big[\partial_\alpha(\widetilde{v}\cdot\partial_t z)-\frac{1}{2}\partial_{\alpha}[|\widetilde{v}|^2]-\frac{1}{\rho_{NS}}\partial_{\alpha}\widetilde{q}-g\partial_\alpha z_2+\frac{\nu_{NS}}{\rho_{NS}}(\Delta v)\cdot\partial_\alpha z\Big]\\
&\hspace{0.2in}-\Big(\frac{\rho_E-\rho_{NS}}{\rho_E}\Big)\partial_t(\widetilde{v}\cdot\partial_\alpha z)\\
&= -\frac{1}{2}\partial_\alpha\Big[\frac{\omega_\Gamma^2}{|\partial_\alpha z|^2}\Big]+\partial_\alpha F+G-\frac{\nu_{NS}}{\rho_E}(\Delta v)\cdot \partial_\alpha z
\end{align*}
where
\begin{align*}
F:=\bigg(\frac{\partial_\alpha z\cdot \partial_t z}{|\partial_\alpha z|^2}-\frac{\widetilde{v}\cdot\partial_\alpha z}{|\partial_\alpha z|^2}\bigg)\omega_\Gamma
\end{align*}
and
\begin{align*}
G&:=\Big(1-\frac{\rho_{NS}}{\rho_E}\Big)\partial_\alpha(\widetilde{v}\cdot \partial_t z)+\Big(\frac{\rho_{NS}-\rho_E}{2\rho_E}\Big)\partial_{\alpha}[|\widetilde{v}|^2]-\frac{1}{\rho_E}\partial_\alpha(\widetilde{q}-\widetilde{p})\\
&\hspace{0.2in}-\Big(\frac{\rho_E-\rho_{NS}}{\rho_E}\Big)\Big(\partial_t(\widetilde{v}\cdot\partial_\alpha z)+g\partial_\alpha z_2\Big)
\end{align*}
and where we have used $$\widetilde{u}\cdot\partial_t z=\widetilde{v}\cdot \partial_t z+\omega_\Gamma\frac{\partial_\alpha z\cdot\partial_t z}{|\partial_\alpha z|^2},$$ as well as $(\widetilde{u}-\widetilde{v})\cdot (\partial_\alpha z)^\perp=0$ via the identity
\begin{align*}
-\frac{\omega_\Gamma^2}{|\partial_\alpha z|^2}+|\widetilde{u}|^2-|\widetilde{v}|^2=\frac{2\omega_\Gamma \widetilde{v}\cdot\partial_\alpha z}{|\partial_\alpha z|^2}.
\end{align*}

Now, fixing $p>1$ and setting
\begin{align*}
I(t):=\int_0^{2\pi} \frac{\omega_\Gamma(t,\alpha)^p}{|\partial_\alpha z|^{p-1}}d\alpha,
\end{align*}
we observe that, using the admissibility criteria
\begin{align*}
|I'(t)|&\lesssim \bigg|\int_0^{2\pi} \frac{p\omega_\Gamma^{p-1}}{|\partial_\alpha z|^{p-1}}\bigg[-\frac{1}{2}\partial_\alpha\bigg(\frac{\omega_\Gamma^2}{|\partial_\alpha z|^2}\bigg)\\
&\hspace{1.2in}+\partial_\alpha F+G\\
&\hspace{1.2in}-\frac{\nu_{NS}}{\rho_E}(\Delta v)(t,z(t,\alpha))\cdot \partial_\alpha z\bigg]d\alpha\bigg|\\
&\hspace{0.2in}+(p-1)\bigg|\int_0^{2\pi} \frac{\omega_\Gamma^p\partial_t|\partial_\alpha z|}{|\partial_\alpha z|^p}d\alpha\bigg|\\
&\lesssim p(I(t)+1)+B(t)
\end{align*}
for all $0\leq t\leq T$, with the implicit constants depending only on $A$ and $T$, with
\begin{align*}
B(t):=\bigg|\int_0^{2\pi} \frac{p\omega_\Gamma^{p-1}}{|\partial_\alpha z|^{p-1}}(-\Delta v)\cdot(\partial_\alpha z)d\alpha\bigg|.
\end{align*}

Using the admissibility hypothesis (\ref{eq-hypB}) followed by a Gronwall argument as in \cite{FIL-NoSplash}, we get the bound $I(t)\lesssim \exp(Cp(T+1))$ for $t\in [0,T]$.  The result follows by taking $p\rightarrow\infty$ (and again using the admissibility criteria).
\end{proof}

We now collect some estimates for the chord-arc functional $CA_T(z)$, where $(u,v,z,p,q)$ is 
an ($A$, $\textrm{E}_{\textrm{wild}}$, $\textrm{NS}_{\textrm{tame}}$)--admissible solution to the 
system ($\ref{eqE1}$)--($\ref{eqE5}$).  

\begin{lemma}
\label{lem1}
Fix $T>0$, $A>0$ and $\delta>0$.  Then there exist $\epsilon_1,\epsilon_2,\epsilon_3>0$ such that if $(u,v,z,p,q)$ is an ($A$, $\textrm{E}_{\textrm{wild}}$, $\textrm{NS}_{\textrm{tame}}$)--admissible solution to the system ($\ref{eqE1}$)--($\ref{eqE5}$), then 
\begin{enumerate}
\item[(i)] for $\alpha,\beta\in\mathbb{R}$, the condition $|\alpha-\beta|\leq \epsilon_1$ implies $F_z(t,\alpha,\beta)\geq \epsilon_1$, 
\item[(ii)] if $t\in [0,T]$ is such that $CA(t,z)\leq \epsilon_2$ then there exist $\alpha_1,\alpha_2\in\mathbb{R}$ such that if $z^{(1)}=z(t,\alpha_1)$ and $z^{(2)}=z(t,\alpha_2)$, then we have
\begin{align*}
2\epsilon_1\leq |\alpha_1-\alpha_2|\leq \frac{1}{2\epsilon_1},
\end{align*}
as well as $\{\lambda z^{(1)}+(1-\lambda)z^{(2)}:\lambda\in [0,1]\}\cap \Gamma_t=\emptyset$, 
\begin{align}
(z^{(1)}-z^{(2)})\cdot z_\alpha(t,\alpha_1)=0,\quad (z^{(1)}-z^{(2)})\cdot z_\alpha(t,\alpha_2)=0,\label{eqza1}
\end{align}
\begin{align*}
z_\alpha(t,\alpha_1)\cdot z_\alpha(t,\alpha_2)<0,
\end{align*}
and
\begin{align*}
|z^{(1)}-z^{(2)}|=\inf_{\{(\alpha,\beta):\epsilon_1\leq |\alpha-\beta|\leq \frac{1}{\epsilon_1}\}}|z(t,\alpha)-z(t,\beta)|\sim CA(t,z), 
\end{align*}
and
\item[(iii)] if $t\in [0,T]$ and $\alpha_1,\alpha_2\in \mathbb{R}$ are such that, for $z^{(1)}=z(t,\alpha_1)$ and $z^{(2)}=z(t,\alpha_2)$, we have $$\{\lambda z^{(1)}+(1-\lambda)z^{(2)}:\lambda\in [0,1]\}\subset \overline{\Omega_{NS}(t)},$$ then the condition $|\alpha_1-\alpha_2|\geq \delta$ implies $|z^{(1)}-z^{(2)}|\geq \epsilon_3$.
\end{enumerate}
\end{lemma}

\begin{proof}
The statements (i) and (ii) 
follow as in Lemma 2.3 in \cite{FIL-NoSplash}, and the proofs given there work equally well in our setting.  

Suppose for contradiction that the claim (iii) failed.  Then there is $\delta>0$ so that for any $\epsilon_3>0$ there exist an admissible solution and $t_0$, $\alpha_1$, $\alpha_2$ with $\{\lambda z^{(1)}+(1-\lambda)z^{(2)}:\lambda\in [0,1]\}\subset \overline{\Omega_{NS}(t_0)}$, and such that $|\alpha_1-\alpha_2|\geq \delta$ and $|z^{(1)}-z^{(2)}|\leq \epsilon_3$, where $z^{(i)}=z(t_0,\alpha_i)$, $i=1,2$.

By the equation for $\partial_t z(t,\alpha)$, we have
\begin{align*}
\partial_t z(t,\alpha)=v(t,z(t,\alpha))
\end{align*}

Fix $s,t\in [0,t_0]$ with $s<t$.  The regularity assumption on $v$ gives, with $z^{(i)}=z(t_0,\alpha_i)$, $i=1,2$,
\begin{align*}
&|z(s,\alpha_1)-z(s,\alpha_2)|\\
&\hspace{0.2in}\leq |z(t,\alpha_1)-z(t,\alpha_2)|+\int_s^{t} |\partial_t[z(t',\alpha_1)-z(t',\alpha_2)]|dt'\\
&\hspace{0.2in}=|z(t,\alpha_1)-z(t,\alpha_2)|+\int_s^{t} |v(t',z(t',\alpha_1))-v(t',z(t',\alpha_2))|dt'\\
&\hspace{0.2in}\leq |z(t,\alpha_1)-z(t,\alpha_2)|+\int_s^{t} \bigg(\sup_{w}|(\nabla_w v)(t',w)|\bigg)|z(t',\alpha_1)-z(t',\alpha_2)|dt\\
&\hspace{0.2in}\leq |z(t,\alpha_1)-z(t,\alpha_2)|+A\int_s^{t} |z(t',\alpha_1)-z(t',\alpha_2)|dt',
\end{align*}
provided that $\{\lambda z(\tau,\alpha_1)+(1-\lambda)z(\tau,\alpha_2):\lambda\in [0,1]\}\subset \overline{\Omega_{NS}(\tau)}$ for all $\tau\in [s,t]$ (indeed, the fact that this is satisfied at time $t_0$ implies that this condition is preserved for all $t\in [0,T]$).

As a consequence, setting $\zeta(s)=\sup_{s'\in [s,t_0]} |z(s',\alpha_1)-z(s',\alpha_2)|$ for $s\in [0,t_0]$, we have
\begin{align*}
\zeta(s)&\leq \zeta(t)+(t-s)A\zeta(s), \quad 0\leq s<t\leq t_0,\\
\zeta(t_0)&\leq \epsilon_3.
\end{align*}
A continuity argument now shows $\zeta(s)\lesssim \epsilon_3$ for all $s\in [0,t_0]$ (in particular, the set of $s\in [0,t_0]$ such that $\zeta(s)\leq 2\epsilon_3$ is both open and closed in $[0,t_0]$), and thus
\begin{align*}
|z(0,\alpha_1)-z(0,\alpha_2)|\lesssim \epsilon_3.
\end{align*}

Recalling also $|\alpha_1-\alpha_2|\geq \delta$, we therefore obtain $CA(0,z)\leq F_z(0,\alpha_1,\alpha_2)=\frac{|z(0,\alpha_1)-z(0,\alpha_2)|}{|\alpha_1-\alpha_2|}\lesssim \epsilon_3\delta^{-1}$.  Choosing $\epsilon_3$ sufficiently small gives a contradiction with the admissibility hypothesis $CA(0,z)\geq 1/A$.  This completes the proof of (iii).

\end{proof}

The remaining two lemmas of this section, which are analogues of \cite[Proposition $2.5$]{FIL-NoSplash}, further characterize the nature of a candidate splash point.  The first lemma asserts that under appropriate a priori assumptions, if $\alpha_1$ and $\alpha_2$ are chosen as in Lemma $\ref{lem1}$, then in a neighborhood of the segment connecting the candidate splash points on the boundary, the domain partitions into a region of the Euler fluid $\Omega_{E}(t)$ ``enclosed'' above and below by two portions of the Navier-Stokes region $\Omega_{NS}(t)$.

\begin{lemma}
\label{lem2}
Fix $T>0$ and $A>0$, and let $\epsilon_2$ be as given by the statement of Lemma \ref{lem1}.  Then there exists $\epsilon_4,\epsilon_5\in (0,\epsilon_2)$ such that if $(u,v,z,p,q)$ is an ($A$, $\textrm{E}_{\textrm{wild}}$, $\textrm{NS}_{\textrm{tame}}$)--admissible solution to the system ($\ref{eqE1}$)--($\ref{eqE5}$) with $CA(t,z)\leq \epsilon_4$ for some $t\in [0,T]$ then for $\alpha_1$ and $\alpha_2$ as in (iii) of Lemma $\ref{lem1}$, setting $z^{(1)}=z(t,\alpha_1)$ and $z^{(2)}=z(t,\alpha_2)$, and $$e:=\frac{z^{(2)}-z^{(1)}}{|z^{(2)}-z^{(1)}|},\quad d:=|z^{(2)}-z^{(1)}|,$$ we have $$\{\lambda z^{(1)}+(1-\lambda) z^{(2)}:\lambda\in [0,1]\}\subset \Omega_{E}(t),$$ and
\begin{align*}
\{z(t,\beta)+\rho e:|\alpha_1-\beta|\leq \epsilon_5,0<\rho<d\}\subset \Omega_{E}(t),\\
\{z(t,\beta)-\rho e:|\alpha_1-\beta|\leq \epsilon_5,0<\rho<\epsilon_5\}\subset \Omega_{NS}(t),
\end{align*}
along with
\begin{align*}
\{z(t,\beta)+\rho e:|\alpha_2-\beta|\leq \epsilon_5,0<\rho<\epsilon_5\}\subset \Omega_{NS}(t),\\
\{z(t,\beta)-\rho e:|\alpha_2-\beta|\leq \epsilon_5,0<\rho<d\}\subset \Omega_{E}(t).
\end{align*}
\end{lemma}

The second lemma gives a characterization of the candidate splash under a rigid-body transformation placing the splash in
a reference orientation.

\begin{lemma}
\label{lem3}
Fix $T>0$ and $A>0$, and let $\epsilon_2$ be as given by the statement of Lemma \ref{lem1}.  Then we can choose $\epsilon_4$ and $\epsilon_5$ sufficiently small so that the condition of Lemma \ref{lem2} is satisfied, and, moreover, there exists $C>0$ such that if $(u,v,z,p,q)$ is and ($A$, $\textrm{E}_{\textrm{wild}}$, $\textrm{NS}_{\textrm{tame}}$)--admissible solution to the system ($\ref{eqE1}$)--($\ref{eqE5}$) with $CA(t,z)\leq \epsilon_4$ for some $t\in [0,T]$ and $\alpha_1,\alpha_2,z^{(1)},z^{(2)},e,$ and $d$ are as in the statement of Lemma \ref{lem2}, then, letting $R:\mathbb{R}^2\rightarrow\mathbb{R}^2$ be the rigid transformation defined by $$R\left(-\frac{d}{2}e_2\right)=z^{(1)},\quad R(0)=\frac{z^{(1)}+z^{(2)}}{2},\quad R\left(\frac{d}{2}e_2\right)=z^{(2)},$$
there exist functions $$f_1,f_2,\beta_1,\beta_2:I_0\rightarrow I_1,\quad I_0=(-\epsilon_5^4,\epsilon_5^4),\quad I_1=(-\epsilon_5^2,\epsilon_5^2),$$ such that $f_i\in C^4(I_0)$ and $\beta_i\in C^4(I_0)$ with $\lVert f_i\rVert_{C^4}+\lVert \beta_i\rVert_{C^4}\leq C$ for $i=1,2$, and 
\begin{enumerate}
\item[(i)] $f_1(0)=-d/2$, $f_2(0)=d/2$, $f'_1(0)=f'_2(0)=0$, with $f_2(\rho)-f_1(\rho)\geq d$ for $\rho\in I_0$,
\item[(ii)] $\beta_1(0)=\beta_2(0)=0$, with $-\beta'_1(\rho)\beta'_2(\rho)\sim 1$ for $\rho\in I_0$, 
\item[(iii)] $R(\rho,f_i(\rho))=z(t,\alpha_i+\beta_i(\rho))$, for $\rho\in I_0$ and $i=1,2$, and
\item[(iv)] setting 
\begin{align*}
A_1:=\{(\rho,y):\rho\in I_0,y\in I_1\cap (-\infty,f_1(\rho))\},\\
A_2:=\{(\rho,y):\rho\in I_0,y\in I_1\cap (f_1(\rho),f_2(\rho))\},\\
A_3:=\{(\rho,y):\rho\in I_0,y\in I_1\cap (f_2(\rho),\infty)\},
\end{align*}
we have $$R(A_i)\subset \Omega_{NS}(t)\,\,\textrm{for}\,\,i\in\{1,3\}\,\,\,\textrm{and}\,\,\,R(A_2)\subset\Omega_{E}(t).$$
\end{enumerate}
\end{lemma}

Having recalled these statements, we are ready to move on to our next task, proving Theorem $\ref{thm1}$.  This is the content 
of the next section.

\section{Proof of the main result}

In this section, we give the proof of our main result, Theorem $\ref{thm1}$, which asserts that splash-type singularities cannot
form in finite time.

The argument is by contradiction, and is based on the a priori estimates established in the previous sections, combined with an analysis
of the singular integral appearing in the contribution of $B_{\Gamma}$ to the system of PDEs describing the evolution of the interface.  The
main point is to establish a differential inequality for the ``distance of closest approach'' in the splash.

\begin{proof}[Proof of Theorem \ref{thm1}]
Suppose that $T>0$, $A>0$ are such that for all $C_0>0$ there exists an ($A$, $\textrm{E}_{\textrm{wild}}$, $\textrm{NS}_{\textrm{tame}}$)--admissible solution $(u,v,z,p,q)$ to the system ($\ref{eqE1}$)--($\ref{eqE5}$) such that $$0<CA_T(z)<C_0.$$

Let $\epsilon_4$ be as in Lemma $\ref{lem3}$, and let $K>0$ be a fixed constant to be determined later in the argument.  By choosing $\epsilon_4$ smaller if necessary, we may suppose $\epsilon_4\leq 1/A$.  Moreover, by our assumption, we may choose an ($A$, $\textrm{E}_{\textrm{wild}}$, $\textrm{NS}_{\textrm{tame}}$)--admissible solution $(u,v,z,p,q)$ so that $$CA_T(z)\leq \frac{\epsilon_4}{K}.$$

Let $(u,v,z,p,q)$ be such a solution, and choose $\overline{t}\in [0,T]$ so that $CA(\overline{t};z)\leq \epsilon_4/K$.  By the admissibility condition, $CA(0,z)\geq 1/A$.  Thus, if $K$ is chosen sufficiently large, the set $$J:=\{t\geq 0:CA(t,z)\geq \epsilon_4/2\}$$ is nonempty with $t_0:=\sup J<\overline{t}$.

For $t\in [t_0,\overline{t}]$, we define 
\begin{align*}
D(t):=\sup\bigg\{\frac{1}{|z(t,\alpha)-z(t,\beta)|^2}:\alpha,\beta\in\mathbb{R},\epsilon_1<|\alpha-\beta|<\frac{1}{\epsilon_1}\bigg\}.
\end{align*}
We will show
\begin{align}
D(s)\leq D(r)+C\int_{r}^{s} D(t')\log D(t')dt',\,\, t_0\leq r\leq s\leq \overline{t}, \label{eq-dt}
\end{align}
for some universal constant $C>0$ (depending only on $T$ and $A$).  To show (\ref{eq-dt}), we note that a continuity argument reduces the issue to showing the claim that for each $t_2\in [0,T]$ there exists $\delta(t_2)>0$ such that 
\begin{align}
\nonumber D(t_2)\leq D(s)+C\int_{s}^{t_2} D(t')\log D(t')dt'\,\,\textrm{for}\,\, s\in I(t_2),\, \textrm{with}\\
I(t_2):=(t_2-\delta(t_2),t_2]\cap [t_0,\overline{t}].\label{eq-rs}
\end{align}

We now establish this claim.  

\vspace{0.1in}

\noindent {\bf Step 1}: Quantifying the evolution of a candidate splash point.

\vspace{0.1in}

Fix $t_2\in [t_0,\overline{t}]$ and let $\alpha_1$ and $\alpha_2$ be as in (iii) of Lemma $\ref{lem1}$ with $t=t_2$.  Moreover, set $d:=|z(t_2,\alpha_1)-z(t_2,\alpha_2)|$, and for all $t\in [0,t_2]$ set $$\widetilde{D}(t):=\frac{1}{|z(t,\alpha_1)-z(t,\alpha_2)|^2}.$$  We then have
\begin{align}
\nonumber\widetilde{D}'(t_2)&=-2\widetilde{D}(t_2)^2\Big[\Big(z(t_2,\alpha_1)-z(t_2,\alpha_2)\Big)\cdot\Big((\partial_t z)(t_2,\alpha_1)-(\partial_t z)(t_2,\alpha_2)\Big)\Big]\\
&=-2\widetilde{D}(t_2)^2\Big[\Big(z(t_2,\alpha_1)-z(t_2,\alpha_2)\Big)\cdot\Big(\widetilde{u}(t_2,\alpha_1)-\widetilde{u}(t_2,\alpha_2)\Big)\Big],\label{eq-dprime}
\end{align}
where we have used (\ref{eqE3}) and the conditions in (\ref{eqza1}) of Lemma \ref{lem1}.\footnote{The relevant computation is given by noting that
\begin{align*}
z_t=\frac{1}{|z_\alpha|^2}\Big((z_t\cdot z_\alpha)z_\alpha+(\widetilde{u}\cdot z_\alpha^\perp)z_\alpha^\perp\Big)=\frac{1}{|z_\alpha|^2}(z_t\cdot z_\alpha-\widetilde{u}\cdot z_\alpha)z_\alpha+\widetilde{u}
\end{align*}
and then observing that (\ref{eqza1}) implies
\begin{align*}
(z(t_2,\alpha_1)-z(t_2,\alpha_2))\cdot z_\alpha(t_2,\alpha_i)=0,\quad i=1,2.
\end{align*}}

On the other hand, examining the second factor on the right-hand side of (\ref{eq-dprime}), we get
\begin{align*}
&\Big(z(t_2,\alpha_1)-z(t_2,\alpha_2)\Big)\cdot\Big(\widetilde{u}(t_2,\alpha_1)-\widetilde{u}(t_2,\alpha_2)\Big)\\
&\hspace{0.2in}=\Big(R(-\frac{d}{2}e_2)-R(\frac{d}{2}e_2)\Big)\cdot \Big(\widetilde{u}(t_2,\alpha_1)-\widetilde{u}(t_2,\alpha_2)\Big)\\
&\hspace{0.2in}=\Big(Q(-\frac{d}{2}e_2)-Q(\frac{d}{2}e_2)\Big)\cdot \Big(\widetilde{u}(t_2,\alpha_1)-\widetilde{u}(t_2,\alpha_2)\Big)\\
&\hspace{0.2in}=d\Big[\Big(Q^{-1}(\widetilde{u}(t_2,\alpha_1))-Q^{-1}(\widetilde{u}(t_2,\alpha_2))\Big)\cdot e_2\Big]
\end{align*}
where we have set $Q(x)=R(x)-R(0)$ and used that $R$ a rigid body transformation implies that $Q$ is an orthogonal linear transformation.

Now, set $I_2=(-\epsilon_5^6,\epsilon_5^6)\subset\mathbb{R}$, and for $i\in \{1,2\}$ and $\rho\in I_2$, set
\begin{align*}
U_i(\rho):=Q^{-1}(\widetilde{u}(t_2,\alpha_i+\beta_i(\rho))),
\end{align*}
\begin{align*}
V_i(\rho):=Q^{-1}(B_{\textrm{NS}}(z,\omega_v;t_2,\alpha_i+\beta_i(\rho))),
\end{align*}
\begin{align*}
W_i(\rho):=Q^{-1}\Big(\frac{\omega_\Gamma\partial_\alpha z}{2|\partial_\alpha z|^2}\bigg|_{(t_2,\alpha_i+\beta_i(\rho))}+B_\Gamma(z,\omega_\Gamma;t_2,\alpha_i+\beta_i(\rho))\Big)
\end{align*}

Note that $U_i=V_i+W_i$ for $i=1,2$.  To estimate $(U_1(0)-U_2(0))\cdot e_2$, it therefore suffices to estimate $V_1(0)-V_2(0)$ and $W_1(0)-W_2(0)$.

\vspace{0.1in}

\noindent {\bf Step 2}: {\it Bounds for $|(V_1(0)-V_2(0))\cdot e_2|$ and $|(W_1(0)-W_2(0))\cdot e_2|$.}

\vspace{0.1in}

We claim that there exists a constant $C$ depending only on $T$, $A$ and the a priori bounds for $\omega_\Gamma$ and $\omega_v$ given by Proposition \ref{lem0} such that we have the bounds
\begin{enumerate}
\item[(i)] $|(W_1(0)-W_2(0))\cdot e_2|\leq Cd|\log(d)|$, and
\item[(ii)] $|(V_1(0)-V_2(0))\cdot e_2|\leq Cd|\log(d)|$.
\end{enumerate}

The bound (i) follows from singular integral estimates obtained in \cite{FIL-NoSplash}; for the sake of completeness, we give an outline of the argument in the appendix.  To establish the bound (ii), we note that
\begin{align*}
&Q(V_1(0)-V_2(0))\\
&\hspace{0.2in}=\int_{\Omega_{NS}(t_2)} \bigg(\frac{(z(t_2,\alpha_1)-y)^\perp}{|z(t_2,\alpha_1)-y|^2}-\frac{(z(t_2,\alpha_2)-y)^\perp}{|z(t_2,\alpha_2)-y|^2}\bigg)\omega_v(t_2,y)dy\\
&\hspace{0.2in}=\int_{\Omega_{NS}(t_2)}\bigg(\frac{(R(0,f_1(0))-y)^\perp}{|R(0,f_1(0))-y|^2}-\frac{(R(0,f_2(0))-y)^\perp}{|R(0,f_2(0))-y|^2}\bigg)\omega_v(t_2,y)dy
\end{align*}
is equal to, after a change of variables $y\mapsto R(y)$, the constant $\det(Q)$ multiplied by
\begin{align*}
&\int_{R^{-1}(\Omega_{NS}(t_2))}\bigg(\frac{(R(0,f_1(0))-R(y))^\perp}{|R(0,f_1(0))-R(y)|^2}-\frac{(R(0,f_2(0))-R(y))^\perp}{|R(0,f_2(0))-R(y)|^2}\bigg)\omega_v(t_2,R(y))dy\\
&=\int_{R^{-1}(\Omega_{NS}(t_2))}\bigg(\frac{(Q(0,f_1(0))-Q(y))^\perp}{|y_1|^2+|f_1(0)-y_2|^2}-\frac{(Q(0,f_2(0))-Q(y))^\perp}{|y_1|^2+|f_2(0)-y_2|^2}\bigg)\omega_v(t_2,R(y))dy.
\end{align*}
It follows that we have the representation
\begin{align*}
&V_1(0)-V_2(0)=\det(Q)^2\int_{R^{-1}(\Omega_{NS}(t_2))} \bigg(\frac{(y_2-f_1(0),y_1)}{|y_1|^2+|f_1(0)-y_2|^2}\\
&\hspace{2.3in}-\frac{(y_2-f_2(0),y_1)}{|y_1|^2+|f_2(0)-y_2|^2}\bigg)\omega_v(t_2,R(y))dy,
\end{align*}
where we have used that $Q$ a $2\times 2$ orthogonal matrix implies $Q^{-1}[(Qy)^\perp]=\det(Q)y^\perp$ for all $y\in\mathbb{R}^2$.

This in turn gives
\begin{align*}
&|(V_1(0)-V_2(0))\cdot e_2|\\
&\hspace{0.2in}=\bigg|\int_{R^{-1}(\Omega_{NS}(t_2))} \frac{y_1((f_2(0)-y_2)^2-(f_1(0)-y_2)^2)}{(y_1^2+(f_1(0)-y_2)^2)(y_1^2+(f_2(0)-y_2)^2)}\omega_v(t_2,R(y))dy\bigg|\\
&\hspace{0.2in}=\bigg|\int_{R^{-1}(\Omega_{NS}(t_2))} \frac{2dy_1y_2\omega_v(t_2,R(y))}{(y_1^2+(d/2+y_2)^2)(y_1^2+(d/2-y_2)^2)}dy\bigg|.
\end{align*}
Set $U:=R^{-1}(\Omega_{NS}(t_2))$ and
\begin{align*}
h_d(z_1,z_2):=\frac{2dz_1z_2\omega_v(t_2,R(z_1,z_2))}{(z_1^2+(\frac{d}{2}-z_2)^2)(z_1^2+(\frac{d}{2}+z_2)^2)},\quad z=(z_1,z_2)\in\mathbb{R}^2.
\end{align*}

We break the integral into the sum of several contributions.  Fix $c>0$ and consider 
\begin{align*}
(I)_{\textrm{far}}:=\int_{U_{\textrm{far}}} h_d(z)dz,\quad (I)_{\textrm{near}}:=\int_{U_{\textrm{near}}} h_d(z)dz,
\end{align*}
where
\begin{align*}
U_{\textrm{far}}&:=U\cap\{z:|z-(0,d/2)|>cd,|z-(0,-d/2)|>cd\},\\
U_{\textrm{near}}&:=U\cap (\{z:|z-(0,d/2)|<cd\}\cup\{z:|z-(0,-d/2)|<cd\}).
\end{align*}

We first estimate $(I)_{\textrm{far}}$.  For this, set $\widetilde{z}:=z-(0,d/2)$, and note that
\begin{align*}
|(I)_{\textrm{far}}|&\lesssim d\int_{U_{\textrm{far}}} \frac{z_1^2|\omega_v|}{|\widetilde{z}|^2|z+(0,d/2)|^2}+\frac{z_2^2|\omega_v|}{|\widetilde{z}|^2|z+(0,d/2)|^2}dz\\
&\lesssim d\bigg(\int_{cd<|\widetilde{z}|<1} \frac{\lVert \omega_v\rVert_{L^\infty}}{|\widetilde{z}|^2}dz+\lVert\omega_v\rVert_{L^2}\bigg(\int_{|\widetilde{z}|\geq 1} |\tilde{z}|^{-4}dz\bigg)^{1/2}\bigg)\\
&\hspace{0.2in}+d\int_{U_{\textrm{far}}} \frac{|(z_2-d/2)(z_2+d/2)+d^2/4|\,|\omega_v|}{|\widetilde{z}|^2|z+(0,d/2)|^2}dz\\
&\lesssim d|\log(cd)|+d\int_{U_{\textrm{far}}} \frac{|\omega_v|}{|\widetilde{z}||\widetilde{z}+(0,d)|}dz
+d\int_{U_{\textrm{far}}} \frac{|\omega_v|}{|\widetilde{z}+(0,d)|^2}dz
\end{align*}
Now,
\begin{align*}
d\int_{U_{\textrm{far}}} \frac{|\omega_v|}{|\widetilde{z}||\widetilde{z}+(0,d)|}dz&\leq d\int_{U_{\textrm{far}}} \frac{|\omega_v|}{|\widetilde{z}|^2}dz+d\int_{U_{\textrm{far}}} \frac{|\omega_v|}{|\widetilde{z}+(0,d)|^2}dz\\
&\lesssim d|\log(cd)|
\end{align*}
while, similarly,
\begin{align*}
d\int_{U_{\textrm{far}}} \frac{|\omega_v|}{|\widetilde{z}+(0,d)|^2}dz\lesssim d|\log(cd)|.
\end{align*}
We therefore have $(I)_{\textrm{far}}\lesssim d|\log(cd)|$.

Turning to the estimate of $(I)_{\textrm{near}}$, fix $c=1/4$.  Then
\begin{align}
\nonumber |(I)_{\textrm{near}}|&\leq \bigg(\int_{R^{-1}(\Omega_{NS}(t_2))\cap\{z:|z-(0,d/2)|<d/4\}} |h_d| dz\\
&\hspace{0.4in}+\int_{R^{-1}(\Omega_{NS}(t_2))\cap\{z:|z-(0,-d/2)|<d/4\}} |h_d| dz\bigg)\label{eq2p}
\end{align}

To estimate the first term on the right-hand side of (\ref{eq2p}), note that
\begin{align*}
\int_{R^{-1}(\Omega_{NS}(t_2))\cap\{z:|z-(0,d/2)|<d/4\}} |h_d|dz&\lesssim d\int_{\{z:|z|<d/4\}} \frac{|z_1(z_2+(d/2))|}{|z|^2|z+(0,d)|^2}dz,
\end{align*}
where we have invoked the change of variables $z\mapsto z+(0,d/2)$, used the a priori bound on $\omega_v$ given by Proposition $\ref{lem0}$, and expanded the region of integration.

We now bound the right-hand side of this expression, via
\begin{align*}
&d\int_{|z|<d/4} \frac{|z_1|(|z_2|+d/2)}{|z|^2(z_1^2+(z_2+d)^2)}dz\\
&\hspace{1.2in}\leq d\int_{|z|<d/4} \frac{1}{z_1^2+(z_2+d)^2}+\frac{(d/2)}{|z|\,(z_1^2+(z_2+d)^2)}dz\\
&\hspace{1.2in}\lesssim d\int_{|z-(0,d)|<d/4} \frac{1}{|z|^2}dz+d^2\int_{|z|<d/4} \frac{dz}{|z|\,(z_1^2+(z_2+d)^2)}.
\end{align*}
The desired estimate now follows from the bounds
\begin{align*}
d\int_{|z-(0,d)|<d/4} \frac{1}{|z|^2}dz&\lesssim d\int_{3d/4<r<5d/4}\int_{0}^\pi r^{-1}drd\theta\lesssim d
\end{align*}
and
\begin{align*}
d^2\int_{|z|<d/4}\frac{dz}{|z|\,(z_1^2+(z_2+d)^2)}&\lesssim \int_{|z|<d/4}\frac{1}{|z|}dz\lesssim \int_0^{d/4} dr\lesssim d,
\end{align*}
where we have used that $|z|<-d/4$ implies $|z_1^2+(z_2+d)^2|\geq (z_2+d)^2\gtrsim d^2$.

Combining the above estimates with analogous bounds for the second term in (\ref{eq2p}), obtained from estimating 
\begin{align*}
d\int_{\{z:|z|<d/4\}} \frac{z_1(z_2-(d/2))}{|z-(0,d)|^2|z|^2}dz
\end{align*}
in an identical manner as above, one obtains $|(I)_{\textrm{near}}|\lesssim d$.  This in turn gives
$$|(I)|\leq |(I)_{\textrm{far}}|+|(I)_{\textrm{near}}|\lesssim d|\log(d)|\quad\textrm{for}\quad 0<d\leq 1,$$
which establishes the desired bound (ii).

\vspace{0.1in}

\noindent {\bf Step 3}: {\it Conclusion of the argument.}

\vspace{0.1in}

Combining the claims (i) and (ii) of Step 2, we get $$|(U_1(0)-U_2(0))\cdot e_2|\leq Cd|\log(d)|.$$  In view of ($\ref{eq-dprime}$) we have shown
\begin{align*}
|\widetilde{D}'(t_2)|\leq C\widetilde{D}(t_2)\log(\widetilde{D}(t_2)).
\end{align*}
Combining this estimate with the assumption $CA_T(z)>0$, we obtain that there exists $\delta(t_2)>0$ such that for $I(t_2)$ as in (\ref{eq-rs}) one has $|\widetilde{D}'(t)|\leq C\widetilde{D}(t)\log(\widetilde{D}(t))$ for all $t\in I(t_2)$.  Integrating both sides of this bound and recalling that by definition one has $\widetilde{D}(s)\leq D(s)$ for $s\in I(t_2)$ then leads to
\begin{align*}
D(t_2)=\widetilde{D}(t_2)\leq \widetilde{D}(s)+C\int_s^{t_2} \widetilde{D}(t')\log\widetilde{D}(t')dt'\leq D(s)+C\int_s^{t_2} D(t')\log D(t')dt'
\end{align*}
for $s\in I(t_2)$.  This establishes ($\ref{eq-rs}$).

Recalling the continuity argument referenced earlier, the estimate ($\ref{eq-dt}$) holds as well.  To conclude the proof of Theorem $\ref{thm1}$, we note that (\ref{eq-dt}) implies $D'(t)\lesssim D(t)\log D(t)$ for $t\in [t_0,\overline{t}]$, and thus
\begin{align*}
D(\overline{t})\leq D(t_0)\exp(\exp(C(\overline{t}-t_0)))\lesssim C(T,A)\epsilon_4^{-2}
\end{align*}
which gives
\begin{align*}
CA(\overline{t},z)\geq c_0(T,A)\epsilon_4
\end{align*}
for a suitable constant $c_0(T,A)>0$.   This contradicts our original choice of $\overline{t}$, provided that $K$ is chosen sufficiently large (depending on $c_0(T,A)$, and thus only on $T$ and $A$).
\end{proof}

\appendix

\section{Singular integral estimates on the Birkhoff-Rott operator $B_\Gamma(z,\omega_\Gamma;t,\alpha)$}

In this appendix we sketch a proof of the claim (i) in Step 2 of the proof of Theorem $\ref{thm1}$.  As we described earlier, the argument is closely related
to estimates obtained in \cite{FIL-NoSplash} for the Birkhoff-Rott operator.  We outline the details for convenience of the reader.

For $\mu\in\mathbb{R}$, set $G(\mu):=(W_1(\mu)-W_2(\mu))\cdot e_2$.  We want to show $|G(0)|\lesssim d|\log(d)|$.  

The first step in the analysis is fix $\epsilon=d/2$, and to observe that the admissibility assumptions imply $|G(\alpha)-G(0)|\lesssim |\alpha|$ for $|\alpha|\lesssim\epsilon$. 

Set $f=-f_1-\epsilon$, $g=f_2-\epsilon$, and let $M$ be the operator defined by
\begin{align*}
(M\omega)(\alpha)&:=(2\pi)^{-1}\int_{\mathbb{R}} \frac{\omega(\beta)(2\epsilon+f(\alpha)+g(\alpha))}{(\alpha-\beta)^2+(2\epsilon+f(\alpha)+g(\alpha))^2}\psi_{\sqrt{\epsilon}/2^6}(\alpha)\psi_{\sqrt{\epsilon}/2^6}(\alpha-\beta)d\beta.
\end{align*}
We will use $M^*$ to denote the adjoint of $M$ with respect to the $L^2$ inner product.  

thus
\begin{align*}
|G(0)|&\lesssim |G(0)|\bigg|\int_{\mathbb{R}} (1+4(M^*)^2)\phi_0d\alpha\bigg|\\
&=\bigg|\int G(0)(1+4(M^*)^2)\phi_0d\alpha\bigg|\\
&\lesssim \bigg| \int G(\alpha)(1+4(M^*)^2)\phi_0d\alpha\bigg|+\int |\alpha|\, |(1+4(M^*)^2)\phi_0|d\alpha,
\end{align*}
where $\phi_0$ is a smooth function compactly supported in the interval $(-\sqrt{\epsilon},\sqrt{\epsilon})$.  With this choice of $\phi_0$, the second term is bounded by a multiple of $d|\log(d)|$, and it suffices to estimate
\begin{align*}
\bigg| \int G\cdot (1+4(M^*)^2)\phi_0d\alpha\bigg|
\end{align*} 

Next, by the arguments described in Section $2.4$ of \cite{FIL-NoSplash}, one has
\begin{align*}
|G(\mu)|&=\bigg|(A-B)(\mu)+(C-D)(\mu)+E(\mu)\bigg| 
\end{align*}
with
\begin{align*}
A(\mu)&:=(2\pi)^{-1}\int_{\mathbb{R}}\frac{\mu-\beta}{(\mu-\beta)^2+(f_1(\mu)-f_2(\beta))^2}\omega_2(\beta)\psi_{\sqrt{\epsilon}}(\mu-\beta)d\beta,\\
B(\mu)&:=(2\pi)^{-1}PV\int_{\mathbb{R}}\frac{\mu-\beta}{(\mu-\beta)^2+(f_2(\mu)-f_2(\beta))^2}\omega_2(\beta)\psi_{\sqrt{\epsilon}}(\mu-\beta)d\beta,\\
C(\mu)&:=-(2\pi)^{-1}\int_{\mathbb{R}}\frac{\mu-\beta}{(\mu-\beta)^2+(f_2(\mu)-f_1(\beta))^2}\omega_1(\beta)\psi_{\sqrt{\epsilon}}(\mu-\beta)d\beta,\\
D(\mu)&:=-(2\pi)^{-1}PV\int_{\mathbb{R}}\frac{\mu-\beta}{(\mu-\beta)^2+(f_1(\mu)-f_1(\beta))^2}\omega_1(\beta)\psi_{\sqrt{\epsilon}}(\mu-\beta)d\beta,\\
E(\mu)&:=\frac{\omega_1(\mu)f'_1(\mu)}{2(1+|f'_1(\mu)|^2)}-\frac{\omega_2(\mu)f'_2(\mu)}{2(1+|f'_2(\mu)|^2)}\\
&\hspace{0.2in}+(Q^{-1}[S_t^\infty(z(\alpha_1+\beta_1(\mu)))-S_t^\infty(z(\alpha_2+\beta_2(\mu)))])\cdot e_2\\
&\hspace{0.2in}+\sum_{m=1}^2 (2\pi)^{-1}PV\int_{\mathbb{R}}\Big(\frac{\mu-\rho}{|\mu-\rho|^2+|f_1(\mu)-f_m(\rho)|^2}\\
&\hspace{0.4in}-\frac{\mu-\rho}{|\mu-\rho|^2+|f_2(\mu)-f_m(\rho)|^2}\Big)\omega_m(\rho)(\psi_{\epsilon_5^{10}}(\beta_m(\rho))-\psi_{\sqrt{\epsilon}}(\mu-\rho))d\rho.
\end{align*}
(In the notation of \cite{FIL-NoSplash}, $A=T_{4,(g,f)}\omega_\Gamma^+(\mu)$, $B=T_{2,f}\omega_\Gamma^+(\mu)$, 
$C=T_{4,(f,g)}\omega_\Gamma^-(\mu)$, and $D=T_{2,g}\omega_\Gamma^-(\mu)$.)

This gives the bound
\begin{align*}
&\bigg|\int G(1+4(M^*)^2)\phi_0d\alpha\bigg|\\
&\hspace{0.2in}\lesssim \bigg|\int \Big((A-B)(\alpha)+(C-D)(\alpha)\Big)\Big((1+4(M^*)^2)\phi_0\Big)(\alpha)d\alpha\bigg|\\
&\hspace{0.4in}+\bigg|\int E(\alpha)\Big((1+4(M^*)^2)\phi_0\Big)(\alpha)d\alpha\bigg|.
\end{align*}
By \cite[Lemma 3.2(ii)]{FIL-NoSplash} and the admissibility assumptions, the second term is bounded by a multiple of $\epsilon|\log(\epsilon)|\lesssim d|\log(d)|$.  On the other hand, the first term is bounded by a multiple of
\begin{align*}
\bigg|\int (H\omega_2-H\omega_1)(\alpha)\Big((1+4(M^*)^2)\phi_0\Big)(\alpha)d\alpha\bigg|.
\end{align*}
where $H$ is the operator defined by 
\begin{align*}
(H\omega)(\alpha)&:=(2\pi)^{-1}\int_{\mathbb{R}} \bigg(\frac{\omega(\beta)(2\epsilon+f(\alpha)+g(\alpha))^2}{(\alpha-\beta)((\alpha-\beta)^2+(2\epsilon+f(\alpha)+g(\alpha))^2)}\\
&\hspace{2.2in}\cdot \psi_{\sqrt{\epsilon}/2^6}(\alpha)\psi_{\sqrt{\epsilon}/2^6}(\alpha-\beta)\bigg)d\beta.
\end{align*}

The desired estimate on this expression now follows by combining \cite[Lemmas 3.2, 3.3 3.5]{FIL-NoSplash} as in the proof of Proposition 3.4 in \cite{FIL-NoSplash} with a priori bounds for the first two terms of the expression $E(\mu)$ obtained as in \cite[(4.5)]{FIL-NoSplash}.

\section{Interfacial regularity estimates for the Navier-Stokes velocity}

In this appendix, we give an estimate for first and second derivatives of the velocity $v$ of the Navier-Stokes fluid, when evaluated at the free interface $\Gamma(t)$, even in the absence of the tameness assumption ($\ref{eq-hypB}$) on $v$.  For this, we introduce a notion of {\it weakly} admissible solution.  This is a variant of the admissibility condition of Definition \ref{def-ad} in which we retain the assumptions ($\ref{eq-assumption}$) and ($\ref{eq-hypC}$) on the regularity of the evolution of the interface but replace the interior regularity assumption ($\ref{eq-hypB}$) with a reduced assumption of regularity on $\tilde{v}$, the restriction of $v$ to the free interface.

\begin{definition}
For $T>0$ and $A>0$, we say that $(u,v,z,p,q)$ is an ($A$, $\textrm{E}_{\textrm{wild}}$, $\textrm{NS}_{\textrm{tame}}$)--weakly admissible 
solution to the system ($\ref{eqE1}$)--($\ref{eqE5}$) if it is a smooth solution to the system satisfying $$CA(0,z)\geq 1/A,\quad  \lVert \widetilde{u}(0,\cdot)\rVert_{L^\infty(\mathbb{R})}\leq A,$$ and $$\lVert v(0)\rVert_{C^3(\Omega_{NS}(0))}\leq A,$$ together
with the conditions (\ref{eq-assumption}) and (\ref{eq-hypC}), but not necessarily (\ref{eq-hypB}).
\end{definition}

We remark that in view of our choice of parametrization $z$, the condition (\ref{eq-assumption}) implies
\begin{align}
\lVert \widetilde{v}\rVert_{C^3([0,T]\times\mathbb{R})}\leq A.\label{eq-assumption-vtilde}
\end{align}

We first show that the condition (\ref{eq-assumption-vtilde}) on the $C^3([0,T]\times\mathbb{R})$ norm of $\widetilde{v}$ implies bounds on all spatial first-order derivatives of $v$ evaluated on the interface.

\begin{lemma}
\label{lem0-0}
Fix $T>0$ and $A>0$, and suppose that $(\omega_\Gamma,v,z,\widetilde{p},q)$ is an ($A$, $\textrm{E}_{\textrm{wild}}$, $\textrm{NS}_{\textrm{tame}}$)--admissible solution of the system ($\ref{eqE1}$)--($\ref{eqE5}$).  Then there exists $C>0$ such that $$\sup_{t\in [0,T]} \lVert \partial_1v(t)\rVert_{L^\infty_x(\Gamma(t))}+\lVert \partial_2v(t)\rVert_{L^\infty_x(\Gamma(t))}\leq C.$$ 
\end{lemma}

\begin{proof}
Fix $t\in [0,T]$.  The admissibility condition implies 
\begin{align*}
\sum_{j=1}^2 |(\partial_1v_j)(\partial_\alpha z_1)+(\partial_2v_j)(\partial_\alpha z_2)|\lesssim A,
\end{align*}
i.e. $|(\nabla v)^\top(\partial_\alpha z)|\lesssim A$.  It therefore suffices to estimate $$(\nabla v)^\top {\mathbf n}=(\nabla v)^\top (-\partial_\alpha z)^\perp=\bigg(\begin{array}{c} (\partial_1v_1)(\partial_\alpha z_2)-(\partial_2v_1)(\partial_\alpha z_1)\\ (\partial_1v_2)(\partial_\alpha z_2)-(\partial_2v_2)(\partial_\alpha z_1)\end{array}\bigg).$$  For this, we use the fact that $v$ is divergence free to write
\begin{align*}
D(v){\mathbf n}
&=-\bigg(\begin{array}{c}-(\partial_1 v_1)(\partial_\alpha z_2)+(\partial_2v_2)(\partial_\alpha z_2)+(\partial_1v_2+\partial_2v_1)(\partial_\alpha z_1)\\
(\partial_1 v_2+\partial_2v_1)(-\partial_\alpha z_2)+(\partial_2v_2)(\partial_\alpha z_1)-(\partial_1v_1)(\partial_\alpha z_1)\end{array}\bigg)
\end{align*}
so that in view of $(\tilde{p}+\tilde{q}+\sigma K_z(\alpha)){\mathbf n}=O(A)$, one obtains
\begin{align*}
|(\partial_2 v_1)(\partial_\alpha z_1)-(\partial_1v_1)(\partial_\alpha z_2)|\lesssim A
\end{align*}
and
\begin{align*}
|(-\partial_1 v_2)(\partial_\alpha z_2)+(\partial_2v_2)(\partial_\alpha z_1)|\lesssim A 
\end{align*}
as desired.
\end{proof}

As a consequence, recalling that $\omega_v$ satisfies $$\partial_t\omega_v+v\cdot \nabla \omega_v=\nu_{NS}\Delta \omega_v$$ in $\{(t,x):t\in [0,T],x\in \Omega_{NS}(t)\}$, a maximum principle argument implies 
\begin{align*}
\lVert \omega_v\rVert_{L^\infty(\{(t,x):t\in [0,T],x\in\overline{\Omega_{NS}(t)}\})}\leq C,
\end{align*}
and
\begin{align*}
\sup_{t\in [0,T]} \lVert \omega_v\rVert_{L_x^p(\Omega_{NS}(t))}\leq C,\quad 1\leq p\leq \infty.
\end{align*}

The second derivative estimates, which arise from Lemma \ref{lem0-0} and differentiation of the compatibility 
condition ($\ref{eqE5}$), are given by the following lemma.

\begin{lemma}
\label{lem0-01}
Fix $T>0$ and $A>0$, and suppose that $(u,v,z,p,q)$ is an ($A$, $\textrm{E}_{\textrm{wild}}$, $\textrm{NS}_{\textrm{tame}}$)--weakly admissible solution of the system ($\ref{eqE1}$)--($\ref{eqE5}$).  Then there exists $C>0$ such that 
\begin{align*}
&\sup_{(t,\alpha)\in [0,T]\times\mathbb{R}} \sum_{j\in\{1,2\}}|\langle (D^2v_j)(t,z(t,\alpha))\,(\partial_\alpha z),\partial_\alpha z\rangle|\\
&\hspace{1.2in}+|\langle (D^2v_j)(t,z(t,\alpha))\, (\partial_\alpha z),(\partial_\alpha z)^\perp\rangle|\leq C.
\end{align*}
\end{lemma}

\begin{proof}
Fix $t\in [0,T]$.  We begin by noting that the hypothesis ($\ref{eq-assumption}$) in the weak admissibility 
condition ensures the regularity estimate (\ref{eq-assumption-vtilde}), namely
\begin{align*}
\lVert \widetilde{v}\rVert_{C^3([0,T]\times\mathbb{R})}\leq A.
\end{align*}
We therefore have
\begin{align*}
&\sum_{j=1}^2 \bigg|(\partial_{11}v_j)(\partial_\alpha z_1)^2+2(\partial_{12}v_j)(\partial_\alpha z_1)(\partial_\alpha z_2)+(\partial_{22}v_j)(\partial_\alpha z_2)^2\\
&\hspace{0.4in}+(\partial_1v_j)(\partial_{\alpha\alpha}z_1)+(\partial_2v_j)(\partial_{\alpha\alpha}z_2)\bigg|\lesssim A,
\end{align*}
where all expressions are evaluated at the point $z(t,\alpha)$ on the interface.  

Moreover, Lemma $\ref{lem0-0}$ also remains valid for ($A$, $\textrm{E}_{\textrm{wild}}$, $\textrm{NS}_{\textrm{tame}}$)--weakly-admissible solutions.  As a consequence, $\lVert \nabla v\rVert_{L^\infty}\lesssim A$, and since $\lVert z(t)\rVert_{C^2(\mathbb{R})} \lesssim A$ by the weak admissibility hypothesis, we have
\begin{align*}
&\sum_{j=1}^2 \bigg|(\partial_{11}v_j)(\partial_\alpha z_1)^2+2(\partial_{12}v_j)(\partial_\alpha z_1)(\partial_\alpha z_2)+(\partial_{22}v_j)(\partial_\alpha z_2)^2\bigg|\lesssim A^2.
\end{align*}

The above bound can be rewritten as $|\langle (D^2v_j)(\partial_\alpha z),\partial_\alpha z\rangle|\lesssim A^2$ for $j=1,2$.  On the other hand, differentiating the compatibility condition (\ref{eqE5}) with respect to $\alpha$ leads to 
\begin{align}
[\partial_\alpha D(v)](\partial_\alpha z)^\perp+D(v)(\partial_{\alpha\alpha}z)^\perp\lesssim A^2.\label{eq-D2-1}
\end{align}
Moreover, we have already bounded the entries in the matrix $D(v)$ and the vector $\partial_{\alpha\alpha}z$ by a constant multiple of $A^2$.  Collecting these estimates, we obtain
\begin{align*}
|[\partial_\alpha D(v)](\partial_\alpha z)^\perp|\lesssim A^2,
\end{align*}
and thus, in view of ($\ref{eq-D2-1}$),
\begin{align}
|[\partial_\alpha D(v)](\partial_\alpha z)^\perp|\lesssim A^2.\label{eq-D2-2}
\end{align}

We now set $$x:=[\partial_\alpha D(v)](\partial_\alpha z)^\perp,$$ and write $x=(x_1,x_2)\in\mathbb{R}^2$.  A computation using the divergence-free condition for $v$ now shows that 
\begin{align*}
x_1&=2[(\partial_{11}v_1)(\partial_\alpha z_1)+(\partial_{12}v_1)(\partial_\alpha z_2)](-\partial_\alpha z_2)\\
&\hspace{0.4in}+\Big[(\partial_{11}v_2)(\partial_\alpha z_1)+(\partial_{12}v_2)(\partial_\alpha z_2)\\
&\hspace{0.6in}+(\partial_{21}v_1)(\partial_\alpha z_1)+(\partial_{22}v_1)(\partial_\alpha z_2)\Big](\partial_\alpha z_1)\\
&=(\partial_{12}v_2)(\partial_\alpha z_1)(\partial_\alpha z_2)+(\partial_{22}v_2)(\partial_\alpha z_2)^2\\
&\hspace{0.4in}+\Big[(\partial_{11}v_1)(\partial_\alpha z_1)+(\partial_{12}v_1)(\partial_\alpha z_2)\Big](-\partial_\alpha z_2)\\
&\hspace{0.6in}+\Big[(\partial_{11}v_2)(\partial_\alpha z_1)+(\partial_{12}v_2)(\partial_\alpha z_2)\\
&\hspace{0.8in}+(\partial_{21}v_1)(\partial_\alpha z_1)+(\partial_{22}v_1)(\partial_\alpha z_2)\Big](\partial_\alpha z_1),
\end{align*}
which gives
\begin{align}
\nonumber x_1&=[(\partial_{11}v_1)(\partial_\alpha z_1)+(\partial_{12}v_1)(\partial_\alpha z_2)](-\partial_\alpha z_2)\\
\nonumber &\hspace{0.4in}+[(\partial_{21}v_1)(\partial_\alpha z_1)+(\partial_{22}v_1)(\partial_\alpha z_2)](\partial_\alpha z_1)+O(A)\\
&=\langle (D^2v_1)(\partial_\alpha z),(\partial_\alpha z)^\perp\rangle+O(A)\label{eq-x1}
\end{align}

Similarly, 
\begin{align*}
x_2&=[(\partial_{11}v_2)(\partial_\alpha z_1)+(\partial_{12}v_2)(\partial_\alpha z_2)\\
&\hspace{0.2in}+(\partial_{21}v_1)(\partial_\alpha z_1)+(\partial_{22}v_1)(\partial_\alpha z_2)](-\partial_\alpha z_2)\\
&\hspace{0.2in}+2[(\partial_{21}v_2)(\partial_\alpha z_1)+(\partial_{22}v_2)(\partial_\alpha z_2)](\partial_\alpha z_1)\\
&=[(\partial_{11}v_2)(\partial_\alpha z_1)+(\partial_{12}v_2)(\partial_\alpha z_2)](-\partial_\alpha z_2)\\
&\hspace{0.2in}-(\partial_{21}v_1)(\partial_\alpha z_1)(\partial_\alpha z_2)-(\partial_{22}v_1)(\partial_\alpha z_2)^2\\
&\hspace{0.2in}+[(\partial_{21}v_2)(\partial_\alpha z_1)+(\partial_{22}v_2)(\partial_\alpha z_2)](\partial_\alpha z_1)\\
&\hspace{0.2in}-(\partial_{11}v_1)(\partial_\alpha z_1)^2-(\partial_{21}v_1)(\partial_\alpha z_2)(\partial_\alpha z_1)
\end{align*}
and thus
\begin{align}
\nonumber x_2&=[(\partial_{11}v_2)(\partial_\alpha z_1)+(\partial_{12}v_2)(\partial_\alpha z_2)](-\partial_\alpha z_2)\\
\nonumber &\hspace{0.2in}+[(\partial_{21}v_2)(\partial_\alpha z_1)+(\partial_{22}v_2)(\partial_\alpha z_2)](\partial_\alpha z_1)+O(A)\\
&=\langle (D^2v_2)(\partial_\alpha z),(\partial_\alpha z)^\perp\rangle+O(A).\label{eq-x2}
\end{align}

Combining ($\ref{eq-D2-2}$) with ($\ref{eq-x1}$) and ($\ref{eq-x2}$) completes the proof.
\end{proof}


\begin{thebibliography}{9}
\bibitem{CS} D. Coutand and S. Shkoller.  On the impossibility of finite-time splash singularities for vortex sheets.  Arch. Ration. Mech. Anal 221 (2016), no. 2, 987--1033.
\bibitem{CS-2} D. Coutand and S. Shkoller.  On the finite-time splash and splat singularities for the 3-D free-surface Euler equations.  Comm. Math. Phys 325 (2014), no. 1, 143--183.
\bibitem{CS-3} D. Coutand and S. Shkoller.  On the splash singularity for the free-surface of a Navier–Stokes fluid.  Ann. Inst. H. Poincare (C) 36 (2019), no. 2, 475--503.
\bibitem{CCFGGS} A. Castro, D. C\'ordoba, C. Fefferman, F. Gancedo, and J. G\'omez-Serrano.  Finite time singularities for the free boundary incompressible Euler equations.  Ann. of Math 178 (2013), 1061--1134.
\bibitem{CCFGGS-2} A. Castro, D. C\'ordoba, C. Fefferman, F. Gancedo, and J. G\'omez-Serrano.  Finite time singularities for water waves with surface tension.  J. Math. Phys 53 (2012), no. 11, 115622.
\bibitem{CCFGGS-3} A. Castro, D. C\'ordoba, C. Fefferman, F. Gancedo, and J. G\'omez-Serrano. Splash singularities for the free boundary Navier-Stokes equations. Ann. PDE 5 (2019), no. 1, Paper No. 12, 117 pp. 
\bibitem{CEG} D. C\'ordoba, A. Enciso, N. Grubic.  Self-intersecting interfaces for stationary solutions of the two-fluid Euler equations.  Ann. PDE 7 (2021), no. 1, Paper No. 12.
\bibitem{FIL-NoSplash} C. Fefferman, A. Ionescu, and V. Lie.  On the absence of splash singularities in the case of two-fluid interfaces.  Duke Math. J. 165 (2016), no. 3, 417--462.
\bibitem{FH} J. Fischer, S. Hensel.  Weak-strong uniqueness for the Navier-Stokes equation for two fluids with surface tension. Arch. Ration. Mech. Anal. 236 (2020), no. 2, 967--1087.
\bibitem{IT} M. Ifrim and D. Tataru. Two dimensional water waves in holomorphic coordinates II: global solutions. Bull. Soc. Math. France 144 (2016), 369--394.
\bibitem{IP} A. Ionescu and F. Pusateri, Global solutions for the gravity water waves system in 2d, Invent. Math. 199 (2015), 653--804. 
\bibitem{Wu} S. Wu. Almost global wellposedness of the 2-D full water wave problem, Invent. Math. 177 (2009), 45--135.
\end{thebibliography}
\end{document}